\documentclass[12pt]{amsart}

\usepackage{a4wide}

\usepackage{enumerate}

\newtheorem{theorem}{Theorem}[section]
\newtheorem{lemma}[theorem]{Lemma}
\newtheorem{proposition}[theorem]{Proposition}

\theoremstyle{definition}

\newcommand{\C}{\ensuremath{\mathbb{C}}}
\newcommand{\R}{\ensuremath{\mathbb{R}}}

\newcommand{\g}[1]{\ensuremath{\mathfrak{#1}}}

\DeclareMathOperator{\Ad}{Ad}
\DeclareMathOperator{\ad}{ad}

\begin{document}
\title[Cohomogeneity one actions on symmetric spaces of rank two]{Cohomogeneity one actions on some noncompact symmetric spaces of rank two}

\author[J.~Berndt]{J\"urgen Berndt}
\address{Department of Mathematics,
King's College London, United Kingdom.}
\email{jurgen.berndt@kcl.ac.uk}

\author[M. Dom\'{\i}nguez-V\'{a}zquez]{Miguel Dom\'{\i}nguez-V\'{a}zquez}
\address{Instituto de Matem\'atica Pura e Aplicada (IMPA), Brazil.}
\email{mvazquez@impa.br}

\thanks{The second author has been supported by a grant from IMPA, and projects MTM2009-07756 (Spanish Government) and GRC2013-045 (Galician Government)}

\subjclass[2010]{Primary 53C35; Secondary 57S20}


\begin{abstract}
We classify, up to orbit equivalence, the cohomogeneity one actions on the noncompact Riemannian symmetric spaces $G_2^{\mathbb C}/G_2$, $SL_3(\C)/SU_3$ and $SO^0_{2,n+2}/SO_2SO_{n+2}$, $n \geq 1$.
\end{abstract}

\keywords{Cohomogeneity one actions, symmetric spaces, parabolic subgroups, exceptional Lie group, real Grassmannian}

\maketitle

\section{Introduction}
A proper isometric action of a Lie group on a Riemannian manifold is of cohomogeneity one if the minimal codimension of its orbits is one. Understanding and classifying these actions on manifolds with a large group of isometries, such as symmetric spaces, is an important problem in Differential Geometry. On the one hand, such study reveals new connections between the geometric and algebraic properties of the symmetric space. On the other hand, this knowledge can be utilized to construct geometric structures on manifolds.

The classification of cohomogeneity one actions on irreducible simply connected compact Riemannian symmetric spaces was completed by Kollross in~\cite{Kollross:tams}, where a list of references  on this problem is included. The classification on reducible symmetric spaces is still open.

In the noncompact setting, the application of the techniques that work for the compact case and the use of the duality of symmetric spaces turn out to fail in general. The classification for rank one symmetric spaces of noncompact type was obtained by Berndt and Tamaru~\cite{BT:tams}, with the exception of the quaternionic hyperbolic spaces, where new examples have been found by D\'iaz-Ramos and Dom\'inguez-V\'azquez~\cite{DD:advances}. For symmetric spaces of higher rank, Berndt and Tamaru studied the case when the orbits form a regular foliation of the ambient space~\cite{BT:jdg}, and the case when there is a totally geodesic singular orbit~\cite{BT:tohoku}.

Recently, Berndt and Tamaru developed a conceptual approach to the classification of cohomogeneity one actions on irreducible symmetric spaces of noncompact type and arbitrary rank \cite{BT:crelle}. By considering the Chevalley and Langlands decompositions of parabolic subgroups of the isometry group of the symmetric space, they were able to show that all cohomogeneity one actions must either appear in the partial classifications in \cite{BT:jdg} and \cite{BT:tohoku}, or arise from two new methods proposed in \cite{BT:crelle}: the canonical extension and the nilpotent construction. The first method allows to construct cohomogeneity one actions on a symmetric space from cohomogeneity one actions on certain totally geodesic submanifolds, the so-called boundary components. Since these boundary components are symmetric spaces of lower rank, the canonical extension suggests a rank reduction approach for the classification. The second method is more intriguing, since only two new examples were found by this technique and not by any other method. Both examples are mysteriously related to the exceptional Lie group $G_2$.

Based on their structure result, Berndt and Tamaru were able to obtain the first complete classifications in some noncompact symmetric spaces of higher rank, namely, in the rank two spaces $SL_3(\R)/SO_3$, $SO^0_{2,3}/SO_2 SO_3$ and $G_2^2/SO_4$. However, their approach seemed to get very complicated when applied to other symmetric spaces.

The purpose of this article is to deepen into the investigation of the techniques introduced in \cite{BT:crelle}, in particular, of the most involved of the two new methods: the nilpotent construction. We will provide some new tools that allow us to obtain the classification of cohomogeneity one actions on the rank two symmetric space $G_2^\C/G_2$, the noncompact dual of the compact exceptional Lie group $G_2$. In view of \cite{BT:crelle}, this space would be the natural candidate to seek a new example produced by the nilpotent construction. However, we prove that this method only leads to the cohomogeneity one action described in~\cite[p.~143]{BT:crelle}. Our result, which is stated in Theorem~\ref{th:G2}, completes the classification of cohomogeneity one actions on symmetric spaces with root system of type $(G_2)$. As an application of these methods we also obtain in Theorem \ref{th:SU3} the classification of cohomogeneity one actions on the symmetric space $SL_3(\C)/SU_3$, the noncompact dual of the compact  Lie group $SU_3$.

Using a more elementary approach, we also derive the classification of cohomogeneity one actions on the indefinite two-plane Grassmannian $SO_{2,n+2}^0/SO_2SO_{n+2}$, $n\geq 1$; see Theorem~\ref{th:Grassmannian} for the classification. Note that the compact dual of this space is the standard oriented real two-plane Grassmannian, which coincides with the nondegenerate complex quadric $Q^{n+2}$ of the complex projective space $\C P^{n+3}$.

Beyond these classifications, we expect that the ideas we introduce in this article will be important in order to complete the study of cohomogeneity one actions on other symmetric spaces, in particular, those with rank two. As pointed out in \cite[p.~132]{BT:crelle}, dealing with spaces of rank higher than two will require some novel approach to understand the cohomogeneity one actions on reducible symmetric spaces.

This paper is organized as follows. Section~\ref{sec:BT} reviews the terminology, notation and previous results needed to understand the rest of the paper. In particular, we state Berndt and Tamaru's main result in Theorem~\ref{th:BT}. In Section~\ref{sec:nilpotent} we develop some new tools to simplify the application of the nilpotent construction method. Finally, in Sections~\ref{sec:G2} and~\ref{sec:Grassmannian} we derive the classification of cohomogeneity one actions on $G_2^\C/G_2$ and $SL_3(\C)/SU_3$, and on the noncompact real two-plane Grassmannians, respectively.

\section{The conceptual result of Berndt and Tamaru} \label{sec:BT}

This section is intended to provide the notation and terminology necessary to understand Berndt and Tamaru's result, as well as the rest of the paper. We follow the notation in \cite{BT:crelle} and we refer to the same article and to \cite{Knapp} for more detailed expositions.

\subsection{Parabolic subgroups}
Let $M=G/K$ be a connected Riemannian symmetric space of noncompact type and rank $r$. Here $G$ is the identity connected component of the isometry group of $M$ and $K$ is the isotropy group of $G$ at a point $o\in M$. As usual, we denote Lie algebras by gothic letters. Thus, let $\g{g}=\g{k}\oplus\g{p}$ be a Cartan decomposition of the real semisimple Lie algebra $\g{g}$ of $G$, where the subspace $\g{p}$ can be identified with the tangent space $T_o M$. Let $\theta$ be the corresponding Cartan involution, given by $\theta (X+Y)=X-Y$ for $X\in\g{k}$ and $Y\in\g{p}$, and $B$ the Killing form of $\g{g}$.  Then $\langle X,Y\rangle=-B(X, \theta Y)$ is a positive definite inner product on $\g{g}$, which satisfies $\langle \ad(X) Y, Z\rangle=-\langle Y,\ad(\theta X)Y\rangle$ for every $X$, $Y$, $Z\in\g{g}$. Henceforth, we will consider $\g{g}$ endowed with this inner product, and denote by $W\ominus V$ the orthogonal complement of $V$ in $W$ with respect to that inner product, for subspaces $V$ and $W$ of $\g{g}$ with $V \subset W$.

Fix a maximal abelian subspace $\g{a}$ of $\g{p}$ and consider the corresponding restricted root space decomposition
$\g{g}=\g{g}_0\oplus \left( \bigoplus_{\alpha\in\Sigma}\g{g}_\alpha\right)$,
where $\Sigma$ is the set of restricted roots, i.e.\ those nonzero covectors $\alpha$ on $\g{a}$ such that  $\g{g}_\alpha=\{X\in\g{g}:[H,X]=\alpha(H)X\text{ for all }H\in\g{a}\}$ is nonzero. It turns out that $\g{g}_0=\g{k}_0\oplus\g{a}$, where $\g{k}_0$ is the centralizer of $\g{a}$ in $\g{k}$. An explicit description of the subalgebra $\g{k}_0$ for each $\g{g}$ can be found in \cite{Tamaru:dga}. For each root $\alpha\in \Sigma$ we define the root vector $H_\alpha\in\g{a}$ by the relation $\alpha(H)=\langle H_\alpha, H\rangle$ for all $H\in\g{a}$, and, for each simple root $\alpha_i$, we also define its corresponding dual vector $H^i\in \g{a}$, determined by the fact that $\alpha_k(H^i)$ is the Kronecker delta $\delta_{ik}$.

Let $r$ be the rank of $M$, $\Lambda=\{\alpha_1,\dots,\alpha_r\}$ a set of simple roots of $\Sigma$, and denote by $\Sigma^+$ the corresponding set of positive roots. Define the nilpotent subalgebra $\g{n}=\bigoplus_{\alpha\in\Sigma^+}\g{g}_\alpha$. Then $\g{g}=\g{k}\oplus\g{a}\oplus\g{n}$ is an Iwasawa decomposition of $\g{g}$.

The conjugacy classes of parabolic subalgebras of $\g{g}$ are parametrized by the subsets $\Phi$ of $\Lambda$. The maximal proper parabolic subalgebras correspond to subsets $\Phi$ with cardinality $r-1$. We will restrict ourselves to describing these maximal proper parabolic subalgebras.
Let $\Phi_j= \Lambda\setminus\{\alpha_j\}$, denote by $\Sigma_j$ the root subsystem of $\Sigma$ generated by $\Phi_j$, and put $\Sigma_j^+=\Sigma_j\cap\Sigma^+$. We define a reductive subalgebra $\g{l}_j$ and a nilpotent subalgebra $\g{n}_j$ of $\g{g}$ by $\g{l}_j=\g{g}_0\oplus\bigl(\bigoplus_{\alpha\in\Sigma_j}\g{g}_\alpha\bigr)$ and $\g{n}_j=\bigoplus_{\alpha\in\Sigma^+\setminus\Sigma^+_j}\g{g}_\alpha$. Let $\g{a}_j=\bigcap_{\alpha\in\Phi_j}\ker\alpha$ and $\g{a}^j=\g{a}\ominus\g{a}_j=\bigoplus_{\alpha\in\Phi_j}\R H_\alpha$. The centralizer and normalizer of $\g{a}_j$ in $\g{g}$ is $\g{l}_j$. Moreover, $[\g{l}_j,\g{n}_j]\subset \g{n}_j$. Then, the Lie algebra $\g{q}_j=\g{l}_j\oplus\g{n}_j$ is the (maximal proper) parabolic subalgebra of $\g{g}$ associated with the subset $\Phi_j$ of $\Lambda$. The decomposition $\g{q}_j=\g{l}_j\oplus\g{n}_j$ is known as the Chevalley decomposition of $\g{q}_j$.

Now we define the reductive subalgebra $\g{m}_j=\g{l}_j\ominus\g{a}_j$ of $\g{g}$. It normalizes $\g{a}_j\oplus\g{n}_j$. Moreover, $\g{g}_j=[\g{m}_j,\g{m}_j]$ is a semisimple subalgebra of $\g{g}$, and the center $\g{z}_j=\g{m}_j\ominus\g{g}_j$ of $\g{m}_j$ is contained in $\g{k}_0$. The decomposition $\g{q}_j=\g{m}_j\oplus\g{a}_j\oplus\g{n}_j$ is called the Langlands decomposition of the parabolic subalgebra $\g{q}_j$. Every maximal proper parabolic subalgebra of $\g{g}$ is conjugate to some of the subalgebras $\g{q}_j$, for some $j\in\{1,\dots,r\}$, by means of an element in $K$.

We will also consider the subalgebra $\g{k}_j$ of $\g{k}$ given by $\g{k}_j=\g{q}_j\cap \g{k}=\g{l}_j\cap\g{k}=\g{m}_j\cap\g{k}=\g{k}_0\oplus\bigl(\bigoplus_{\alpha\in\Sigma_j}\g{k}_\alpha\bigr)$, where $\g{k}_\alpha=\g{k}\cap(\g{g}_{-\alpha}\oplus\g{g}_\alpha)$. Then $[\g{k}_j,\g{n}_j]\subset\g{n}_j$. Moreover, by defining the Lie triple system $\g{b}_j=\g{m}_j\cap\g{p}=\g{g}_j\cap\g{p}$, it turns out that $\g{g}_j=(\g{g}_j\cap\g{k}_j)\oplus\g{b}_j$ is a Cartan decomposition of the semisimple Lie algebra $\g{g}_j$, and $\g{a}^j$ is a maximal abelian subspace of~$\g{b}_j$.

Now we consider some groups associated with the Lie algebras described so far. We let $A$, $N$, $N_j$ and $G_j$ be the connected subgroups of $G$ with Lie algebras $\g{a}$, $\g{n}$, $\g{n}_j$ and $\g{g}_j$, respectively. If we define the reductive group $L_j$ as the centralizer of $\g{a}_j$ in $G$, then $Q_j=L_jN_j$ is the maximal proper parabolic subgroup of $G$ associated with the subset $\Phi_j$ of $\Lambda$. We also define $K_j=L_j\cap K$ and $M_j=K_jG_j$. Then $M_j$ is a closed reductive subgroup of $L_j$, $K_j$ is a maximal compact subgroup of $M_j$, and the center $Z_j$ of $M_j$ is a compact subgroup of $K_j$.

The orbit $B_j=G_j\cdot o$ of the $G_j$-action on $M=G/K$ through $o$ is a connected totally geodesic submanifold of $M$ with $T_o B_j\cong\g{b}_j$. $B_j$ is itself a symmetric space of noncompact type and rank $r-1$, and is called a boundary component of $M$. Moreover, $B_j=G_j\cdot o=M_j\cdot o\cong G_j/(G_j\cap K_j)\cong M_j/K_j$.

Finally, we have an analytic diffeomorphism $M_j\times A_j\times N_j\to Q_j$ which induces an analytic diffeomorphism $B_j\times A_j\times N_j\to M$, $(m\cdot o, a, n)\mapsto (man)\cdot o$, known as a horospherical decomposition of the symmetric space $M$. Note that the Lie triple system $\g{a}_j\cong\R$ determines a geodesic $A_j\cdot o$ in $M$, and we have that $L_j\cdot o\cong B_j\times (A_j\cdot o)$.

\subsection{Classes of cohomogeneity one actions} \label{subsec:classes}
Now we describe the different types of cohomogeneity one actions that appear in the structure result of Berndt and Tamaru.

Let $\ell$ be a one-dimensional subspace of $\g{a}$. Then the connected subgroup $H_\ell$ of $G$ with Lie algebra $\g{h}_\ell=(\g{a}\ominus\ell)\oplus\g{n}$ acts on $M$ with cohomogeneity one, giving rise to a Riemannian foliation whose orbits are congruent to each other. Two choices $\ell$ and $\ell'$ yield orbit equivalent actions if and only if there exists a symmetry of the Dynkin diagram of $\Sigma$ whose corresponding automorphism of $\g{a}$ maps $\ell$ to $\ell'$. See \cite{BT:jdg} for more details.

Now let $\ell$ be a one-dimensional subspace of a simple root space $\g{g}_{\alpha_j}$. Then, the connected subgroup $H_j$ of $G$ with Lie algebra $\g{h}_j=\g{a}\oplus(\g{n}\ominus\ell)$ acts on $M$ with cohomogeneity one, and the orbits form a Riemannian foliation with exactly one minimal leaf. Two choices $\ell\subset \g{g}_{\alpha_j}$ and $\ell'\subset \g{g}_{\alpha_k}$ yield orbit equivalent actions if and only if there is a Dynkin diagram symmetry mapping $\alpha_j$ to $\alpha_k$. Again, see \cite{BT:jdg} for details.

Let $L$ be a maximal proper reductive subgroup of $G$. If $H$ is a subgroup of $L$ acting on $M$ with cohomogeneity one, then the actions of $H$ and $L$ are orbit equivalent and have a totally geodesic orbit, which is singular if $M$ is irreducible and different from a real hyperbolic space. These actions on an irreducible noncompact symmetric space $M$ have been classified in~\cite{BT:tohoku}.

Consider now the Langlands decomposition $Q_j=M_jA_jN_j$ of a maximal proper parabolic subgroup $Q_j$ of $G$ obtained by the choice of the subset $\Phi_j=\Lambda\setminus\{\alpha_j\}$ of $\Lambda$. The corresponding boundary component $B_j$ is a noncompact symmetric space of rank $r-1$ embedded in $M$ as a totally geodesic submanifold. If $H_{\Phi_j}$ is a connected subgroup of the isometry group of $B_j$ acting on $B_j$ with cohomogeneity one, then $H_j^\Lambda=H_{\Phi_j} A_j N_j$ is a connected subgroup of  $Q_j$ acting on $M$ with cohomogeneity one. We say that this action has been obtained by canonical extension of a cohomogeneity one action on the boundary component $B_j$. If two connected closed subgroups $H_{\Phi_j}$, $H_{\Phi_j}'$ of the isometry group $I(B_j)$ of $B_j$ act on $B_j$ with cohomogeneity one and their actions are orbit equivalent by an isometry in the identity component $I^0(B_j)$, then their canonical extensions to $M$ are orbit equivalent as well. More details can be found in \cite{BT:crelle}.

Finally, we describe the so-called nilpotent construction method, which was introduced in \cite{BT:crelle} and will be of fundamental relevance for this work. Consider the Chevalley decomposition $Q_j=L_jN_j$ of a maximal proper parabolic subgroup $Q_j$ of $G$, and recall that $L_j=M_jA_j$. The dual vector $H^j\in \g{a}$ of the simple root $\alpha_j$ induces a gradation $\bigoplus_{\nu\geq 1}\g{n}^\nu_j$ of $\g{n}_j$, where $\g{n}_j^\nu$ is the sum of all root spaces corresponding to positive roots $\alpha\in \Sigma^+\setminus\Sigma_j^+$ with $\alpha(H^j)=\nu$. Let $\g{v}$ be a subspace of $\g{n}^1_j$ with dimension at least $2$. Then $\g{n}_{j,\g{v}}=\g{n}_j\ominus\g{v}$ is a subalgebra of $\g{n}$. Let $N_{j,\g{v}}$ be the corresponding connected subgroup of $N_j$. Denote by $\Theta$ the Cartan involution of $G$ associated with $\theta$, and by the superindex $\cdot^0$ the identity connected component of a group. Assume that:
\begin{enumerate}[{\rm (i)}]
\item $N^0_{L_j}(\g{n}_{j,\g{v}})=\Theta N^0_{L_j}(\g{v})$ acts transitively on $B_j\times (A_j\cdot o)$, and
\item $N^0_{K_j}(\g{n}_{j,\g{v}})=N^0_{K_j}(\g{v})$ acts transitively on the unit sphere of $\g{v}$.
\end{enumerate}
Then $H_{j,\g{v}}=N^0_{L_j}(\g{n}_{j,\g{v}})N_{j,\g{v}}$ is a connected subgroup of $Q_j$ which acts on $M$ with cohomogeneity one and singular orbit $H_{j,\g{v}}\cdot o$. Moreover, if $\g{v}$ and $\g{v}'$ are two such subspaces which are conjugate by an element in $K_j$, then the cohomogeneity one actions by $H_{j,\g{v}}$ and $H_{j,\g{v}'}$ on $M$ are orbit equivalent.

We are now ready to state the main result of \cite{BT:crelle}, which guarantees that all cohomogeneity one actions on irreducible symmetric spaces of noncompact type can be obtained by one of the five methods described above.

\begin{theorem}\cite{BT:crelle}\label{th:BT}
Let $M=G/K$ be a connected irreducible Riemannian symmetric space of noncompact type and rank $r$, and let $H$ be a connected subgroup of $G$ acting on $M$ with cohomogeneity one. Then one of the following statements holds:
\begin{enumerate}[{\rm (1)}]
\item The orbits form a Riemannian foliation on $M$ and one of the following two cases holds:
\begin{enumerate}[{\rm (i)}]
\item The $H$-action is orbit equivalent to the action of $H_\ell$ for some one-dimensional subspace $\ell$ of $\g{a}$.
\item The $H$-action is orbit equivalent to the action of $H_j$ for some $j\in\{1,\dots, r\}$.
\end{enumerate}
\item There exists exactly one singular orbit and one of the following two cases holds:
\begin{enumerate}[{\rm (i)}]
\item $H$ is contained in a maximal proper reductive subgroup $L$ of $G$, the actions of $H$ and $L$ are orbit equivalent, and the singular orbit is totally geodesic.
\item $H$ is contained in a maximal proper parabolic subgroup $Q_j$ of $G$ and one of the following two subcases holds:
\begin{enumerate}[{\rm (a)}]
\item The $H$-action is orbit equivalent to the canonical extension of a cohomogeneity one action with a singular orbit on the boundary component $B_j$ of $M$.
\item The $H$-action is orbit equivalent to the action of a group $H_{j,\g{v}}$ obtained by nilpotent construction, for some subspace $\g{v}\subset\g{n}^1_j$ with $\dim \g{v}\geq 2$.
\end{enumerate}
\end{enumerate}
\end{enumerate}
\end{theorem}

\section{The nilpotent construction}\label{sec:nilpotent}

As pointed out in \cite{BT:crelle}, the application of the nilpotent construction method in symmetric spaces of rank greater than one seems to be quite difficult to deal with. In fact, only two examples of cohomogeneity one actions on symmetric spaces of rank at least two have been obtained by this method and not by any other method (see~\cite[p.~143]{BT:crelle}). Hence, it seems reasonable to deepen into the study of this method by trying to simplify its application and by providing new tools to use it. This is the purpose of this section.

We will use the notation and terminology introduced in Section~\ref{sec:BT} and, in particular, the notation involved in the description of the nilpotent construction method. As above, we fix a subset $\Phi_j=\Lambda\setminus\{\alpha_j\}$ of $\Lambda$, for some $j\in\{1,\dots,r\}$. We start with a general lemma.

\begin{lemma}\label{lemma:ad}
We have that $\ad(H)X=\nu\alpha_j(H)X$, for each $H\in\g{a}_j$ and $X\in\g{n}_j^\nu$.
\end{lemma}
\begin{proof}
Let $H\in\g{a}_j$ and $X\in\g{n}_j^\nu$. Then we can write $X=\sum X_\alpha$,
where $X_\alpha\in\g{g}_\alpha$ and $\alpha$ ranges over all roots in $\Sigma^+\setminus\Sigma^+_j$ such that $\alpha(H^j)=\nu$. Note that, for each one of such roots $\alpha$, there exist integers $x_{\alpha,k}$ such that $\alpha=\nu\alpha_j+\sum_{k\neq j} x_{\alpha,k} \alpha_k$. Then
\[
[H,X]=\sum_{\begin{subarray}{c}\alpha\in\Sigma^+\setminus\Sigma^+_j\\ \alpha(H^j)=\nu\end{subarray}}[H,X_\alpha]= \sum_{\begin{subarray}{c}\tiny\alpha\in\Sigma^+\setminus\Sigma^+_j\\ \alpha(H^j)=\nu\end{subarray}}\Big(\nu\alpha_j(H)+\sum_{k\neq j}x_{\alpha,k}\alpha_k(H)\Big)X_\alpha=\nu\alpha_j(H)X,
\]
where in the last equality we used that $H\in\g{a}_j=\bigcap_{k\neq j}\ker \alpha_k$.
\end{proof}

Now we can prove a result that simplifies the application of the nilpotent construction.

\begin{proposition}\label{prop:subtle}
Assume that $\dim \g{n}^1_j\geq 2$ and let $\g{v}$ be a linear subspace of $\g{n}^1_j$ with $\dim \g{v}\geq 2$ such that
\begin{enumerate}[{\rm (i)}]
\item $N^0_{M_j}(\g{n}_{j,\g{v}})=\Theta N^0_{M_j}(\g{v})$ acts transitively on $B_j$, and
\item $N^0_{K_j}(\g{n}_{j,\g{v}})=N^0_{K_j}(\g{v})$ acts transitively on the unit sphere of $\g{v}$.
\end{enumerate}
Then $H_{j,\g{v}}=N^0_{L_j}(\g{n}_{j,\g{v}})N_{j,\g{v}}$ acts on $M$ with cohomogeneity one and $H_{j,\g{v}}\cdot o$ is a singular orbit of this action. Moreover, if $\g{v}_1$ and $\g{v}_2$ are two such subspaces which are conjugate by an element of $K_j$, then the cohomogeneity one actions of $H_{j,\g{v}_1}$ and $H_{j,\g{v}_2}$ on $M$ are orbit equivalent.
\end{proposition}
\begin{proof}
The result will follow from the hypotheses of the nilpotent construction method (see \S\ref{subsec:classes} above, or \cite[Proposition~4.3]{BT:crelle}) once we show that condition (i) in the statement is equivalent to the fact that $N^0_{L_j}(\g{n}_{j,\g{v}})$ acts transitively on $F_j=(M_j\cdot o)\times(A_j\cdot o)\cong B_j\times\R$.

First, assume that $N^0_{M_j}(\g{n}_{j,\g{v}})$ acts transitively on $B_j$. Lemma~\ref{lemma:ad} implies that $A_j$ normalizes $\g{n}_{j,\g{v}}$, and hence $A_j\subset N^0_{L_j}(\g{n}_{j,\g{v}})$, which together with the hypothesis means that $N^0_{L_j}(\g{n}_{j,\g{v}})$ acts transitively on $F_j$, as desired.

Conversely, assume now that $N^0_{L_j}(\g{n}_{j,\g{v}})$ acts transitively on $B_j\times(A_j\cdot o)$.
Let $p\in B_j$. By hypothesis and since $L_j$ is the direct product $M_j\times A_j$, there is an element $am\in N^0_{L_j}(\g{n}_{j,\g{v}})$, with $a\in A_j$ and $m\in M_j$, such that $(am)\cdot o=p$. But because of the diffeomorphism $M\cong B_j\times A_j\times N_j$ given by the horospherical decomposition, the element $a$ must be the identity, and hence $p=m\cdot o$ with $m\in N^0_{L_j}(\g{n}_{j,\g{v}})\cap M_j$. Therefore $N^0_{M_j}(\g{n}_{j,\g{v}})$ acts transitively on $B_j$.
\end{proof}

We conclude this section with a result that will help us to determine all subspaces $\g{v}$ of $\g{n}^1_j$ satisfying the conditions of Proposition~\ref{prop:subtle}. First, we need the following lemma.

\begin{lemma}\label{lemma:splitting}
Every maximal proper subalgebra $\tau$ of ${\mathfrak m}_j$ can be written as a direct sum 
\[ \tau = \pi_1(\tau) \oplus \pi_2(\tau) \subset {\mathfrak g}_j \oplus {\mathfrak z}_j ,\]
where $\pi_1 \colon {\mathfrak m}_j \to {\mathfrak g}_j$, $\pi_2 \colon {\mathfrak m}_j \to {\mathfrak z}_j$ are the canonical orthogonal projection maps.
\end{lemma}

\begin{proof}
Since the decomposition $ {\mathfrak m}_j = {\mathfrak g}_j \oplus {\mathfrak z}_j $ is a direct sum of Lie algebras, the two projections $\pi_1$ and $\pi_2$ are Lie algebra homomorphisms. Let $\tau$ be a maximal proper subalgebra of ${\mathfrak m}_j$. Then $\pi_1(\tau)$ is a subalgebra of ${\mathfrak g}_j$ and $\pi_2(\tau)$ is a subalgebra of ${\mathfrak z}_j$, and therefore $\pi_1(\tau) \oplus \pi_2(\tau)$
is a subalgebra of ${\mathfrak g}_j \oplus {\mathfrak z}_j = {\mathfrak m}_j$. We obviously have
\[ \tau \subset \pi_1(\tau) \oplus \pi_2(\tau), \]
and since $\tau$ is maximal and proper in ${\mathfrak m}_j$ we must have either
\[\qquad \tau = \pi_1(\tau) \oplus \pi_2(\tau) \qquad \text{or} \qquad {\mathfrak m}_j = \pi_1(\tau) \oplus \pi_2(\tau). \]
In the first case we get a direct sum decomposition of $\tau$ into subalgebras of ${\mathfrak g}_j$ and ${\mathfrak z}_j$. In the second case we get $\pi_1(\tau) = {\mathfrak g}_j$. Since the derived subalgebra $[\tau,\tau]$ of $\tau$ is contained in ${\mathfrak g}_j$ we have
\[ [\tau,\tau] = \pi_1[\tau,\tau] = [\pi_1(\tau),\pi_1(\tau)] = [{\mathfrak g}_j,{\mathfrak g}_j] = {\mathfrak g}_j ,\]
using the fact that ${\mathfrak g}_j$ is semisimple.  This implies $\pi_1(\tau) = {\mathfrak g}_j \subset \tau$ and therefore also $\pi_2(\tau) \subset \tau$. Altogether this gives ${\mathfrak m}_j = \pi_1(\tau) \oplus \pi_2(\tau) \subset \tau$, which contradicts the assumption that $\tau$ is a proper subalgebra of ${\mathfrak m}_j$. 
\end{proof}

Recall that, if $r\geq 2$, then the symmetric space $B_j=G_j/(G_j\cap K_j)$ has rank $r-1$. We assume now that $\g{g}_j$ does not have any nonzero compact ideal. Then the set $\Lambda_j = \Lambda \setminus \{\alpha_j\}$ can be regarded as a set of simple roots for the semisimple Lie algebra $\g{g}_j=(\g{g}_j\cap\g{k}_j)\oplus\g{b}_j$ with respect to the maximal abelian subspace ${\mathfrak a}^j \subset {\mathfrak b}_j$.
Every maximal proper parabolic subalgebra of $\g{g}_j$ is  conjugate via an element of $G_j\cap K_j$ to some of the $r-1$ parabolic subalgebras of $\g{g}_j$ determined by some subset of $\Lambda_j$ of the form $\Lambda_j \setminus\{\alpha_l\}$ for some $l\in\{1,\dots,r\}$, $l\neq j$. We will denote by $\g{q}_{j,l}$ the corresponding parabolic subalgebra of $\g{g}_j$. Note that $\g{q}_{j,l}$ is the intersection of $\g{g}_j$ and the parabolic subalgebra of $\g{g}$ corresponding to $\Lambda \setminus \{\alpha_j,\alpha_l\}$.

Moreover, let us define $\mathcal{V}$ as the set of linear subspaces $\g{v}$ of $\g{n}^1_j$ satisfying the conditions of Proposition~\ref{prop:subtle}. For each $l\in\{1,\dots,r\}$, $l\neq j$, we also define $\mathcal{V}_l$ as the subset of $\mathcal{V}$ given by all subspaces $\g{v}$ such that $N_{\g{m}_j}(\g{v})=\theta N_{\g{m}_j}(\g{n}_{j,\g{v}})$ is contained in $\g{q}_{j,l}\oplus\g{z}_j$. 

\begin{proposition}\label{prop:v}
Assume that the adjoint representation of $\g{g}_j\cap\g{k}_j$ on $\g{n}^1_j$ is irreducible, and $\g{g}_j$ has no nonzero compact ideals. Then, with the notation above:
\[
\mathcal{V}\setminus\{\g{n}^1_j\}=\bigcup_{\begin{subarray}{c}l=1\\l\neq j\end{subarray}}^r \Ad(K_j)\mathcal{V}_l.
\]
\end{proposition}

\begin{proof}
Let $\g{v}\in\mathcal{V}$. Then, thanks to Lemma~\ref{lemma:splitting}, either $N_{\g{m}_j}(\g{n}_{j,\g{v}})\supset\g{g}_j$ or $N_{\g{m}_j}(\g{n}_{j,\g{v}})$ is contained in $\tau=\hat{\tau}\oplus\g{z}_j$, for some maximal proper subalgebra $\hat{\tau}$ of $\g{g}_j$. In the first case we have that $N_{\g{k}_j}(\g{n}_{j,\g{v}})=N_{\g{m}_j}(\g{n}_{j,\g{v}})\cap\g{k}_j\supset\g{g}_j\cap\g{k}_j$, so condition (ii) in Proposition~\ref{prop:subtle} can only be satisfied if $\g{v}=\g{n}^1_j$, since the action of $\g{g}_j\cap\g{k}_j$ on $\g{n}^1_j$ is irreducible by assumption. Hence, let us assume that $N_{\g{m}_j}(\g{n}_{j,\g{v}})$ is contained in $\tau=\hat{\tau}\oplus\g{z}_j$, for some maximal proper subalgebra $\hat{\tau}$ of $\g{g}_j$. It is known that every maximal proper subalgebra of a semisimple real Lie algebra is either reductive or parabolic (see for example \cite[pp.~192--193]{Onishchik}). 

Let $\hat{\tau}$ be a reductive subalgebra of $\g{g}_j$. Since $N_{\g{m}_j}(\g{n}_{j,\g{v}})\subset\hat{\tau}\oplus\g{z}_j$ and $Z_j$ acts trivially on $B_j$, it turns out that $N_{M_j}(\g{n}_{j,\g{v}})$ cannot act transitively on $B_j$ according to the assumption that $\g{g}_j$ has no nonzero compact ideals and \cite[Proposition~3.1]{BT:crelle}. But this contradicts condition~(i) in Proposition~\ref{prop:subtle}.

Therefore, $N_{\g{m}_j}(\g{v})=\theta N_{\g{m}_j}(\g{n}_{j,\g{v}})$ is contained in $\Ad(k)\g{q}_{j,l}\oplus\g{z}_j$, for some $k\in G_j\cap K_j$, and where $\g{q}_{j,l}$ is a fixed maximal proper parabolic subalgebra of $\g{g}_j$.

Define $\tilde{\g{v}}=\Ad(k^{-1})\g{v}$. Then $N_{\g{m}_j}(\tilde{\g{v}})=\Ad(k^{-1})N_{\g{m}_j}(\g{v})\subset \g{q}_{j,l}\oplus\g{z}_j$. Moreover, $\Theta N_{M_j}^0(\tilde{\g{v}})=k^{-1}\Theta N_{M_j}^0(\g{v})k$ acts transitively on $B_j$, and $N_{K_j}(\tilde{\g{v}})=k^{-1}N_{K_j}(\g{v})k$ acts transitively on the unit sphere of $\tilde{\g{v}}$. Hence we get that $\tilde{\g{v}}\in \mathcal{V}_{l}$, and thus $\g{v}\in\Ad(k)\mathcal{V}_{l}$. One of the inclusions of the assertion in the lemma is then proved.

Now, let $\tilde{\g{v}}\in\mathcal{V}_{l}$ and $k\in K_j$. Since $K_j$ normalizes $M_j$, we have that $\Theta N_{M_j}(\Ad(k)\tilde{\g{v}})=k \Theta N_{M_j}(\tilde{\g{v}})k^{-1}$, and $N_{K_j}(\Ad(k)\tilde{\g{v}})=k N_{K_j}(\tilde{\g{v}})k^{-1}$, so $\Ad(k)\tilde{\g{v}}$ satisfies conditions (i) and (ii). Moreover, $\tilde{\g{v}}$ (or, equivalently, $\Ad(k)\tilde{\g{v}}$) cannot be $\g{n}^1_j$, because this would imply that $N_{\g{m}_j}(\tilde{\g{v}})=\g{m}_j$, contradicting the fact that $N_{\g{m}_j}(\tilde{\g{v}})\subset \g{q}_{j,l}\oplus \g{z}_j(\neq \g{m}_j)$ by definition of $\mathcal{V}_{l}$. Hence $\Ad(k)\tilde{\g{v}}\in\mathcal{V}\setminus\{\g{n}^1_j\}$.
\end{proof}

Proposition~\ref{prop:v} provides us with a more manageable method to determine all the subspaces $\g{v}$ of $\g{n}_j^1$ which give rise to cohomogeneity one actions via the nilpotent construction technique. Moreover, according to the last claim in Proposition~\ref{prop:subtle}, all subspaces $\g{v}\in\Ad(K_j)\mathcal{V}_{l}$ give rise to orbit equivalent actions. This means that the moduli space of cohomogeneity one actions up to orbit equivalence obtained by nilpotent construction from the choice $\Phi=\Phi_j$ can be identified with some subset of $\{\g{n}^1_j\}\cup\bigcup_{\begin{subarray}{c}l=1\\l\neq j\end{subarray}}^n \mathcal{V}_l$.

\section{The classifications in $G_2^\C/G_2$ and $SL_3(\C)/SU_3$}\label{sec:G2}

In this section we classify, up to orbit equivalence, the cohomogeneity one actions on the noncompact duals of the compact Lie groups $G_2$ and $SU_3$.

The symmetric space $M=G_2^\C/G_2$ has rank $2$ and dimension $14$. Its root system $\Sigma$ is of type $(G_2)$ and can be identified with the root system of the complex simple exceptional Lie algebra $\g{g}_2^\C$, so all root spaces have complex dimension $1$. Let $\Lambda=\{\alpha_1,\alpha_2\}$ be a set of simple roots, where $\alpha_1$ is the shortest simple root. Then $\Sigma^+=\{\alpha_1,\alpha_2,\alpha_1+\alpha_2, 2\alpha_1+\alpha_2, 3\alpha_1+\alpha_2, 3\alpha_1+2\alpha_2\}$. The maximal abelian subalgebra $\g{a}$ has dimension $2$ and is spanned by the root vectors $H_{\alpha_1}$ and $H_{\alpha_2}$. Moreover, $\g{k}_0=\R iH_{\alpha_1}\oplus\R i H^2=\R i H^1\oplus\R i H_{\alpha_2}\cong\g{u}_1\oplus\g{u}_1$, where $i$ is the complex structure of $\g{g}_2^\C$.

We can now state and prove the classification result for $G_2^\C/G_2$.

\begin{theorem}\label{th:G2}
Each cohomogeneity one action on $M=G_2^\C/G_2$ is orbit equivalent to one of the following cohomogeneity one actions on $M$:
\begin{enumerate} [{\rm (1)}]
\item The action of the subgroup $H_\ell$ of $G_2^\C$ with Lie algebra
\[
\g{h}_\ell=(\g{a}\ominus\ell)\oplus\g{n},
\]
where $\ell$ is a one-dimensional linear subspace of $\g{a}$. The orbits are isometrically congruent to each other and form a Riemannian foliation on $M$.
\item The action of the subgroup $H_{j}$, $j\in\{1,2\}$, of $G_2^\C$ with Lie algebra
\[
\g{h}_{j}=\g{a}\oplus(\g{n}\ominus\ell_j),
\]
where $\ell_j$ is a one-dimensional linear subspace of $\g{g}_{\alpha_j}$. The orbits form a Riemannian foliation on $M$ and there is exactly one minimal orbit.
\item The action of $SL_3(\C)\subset G_2^\C$, which has a totally geodesic singular orbit isometric to the symmetric space $SL_3(\C)/SU_3$.
\item The action of the subgroup $H_{j,0}^\Lambda$, $j\in\{1,2\}$, of $G_2^\C$ with Lie algebra
\[
\g{h}_{j,0}^\Lambda=\g{k}_{\alpha_{j+1}}\oplus\R i H_{\alpha_{j+1}}\oplus(\g{a}\ominus\R H_{\alpha_{j+1}})\oplus(\g{n}\ominus\g{g}_{\alpha_{j+1}}),
\]
where indices are taken modulo $2$ and $\g{k}_{\alpha_{j+1}}\oplus\R i H_{\alpha_{j+1}}=\g{g}_j\cap\g{k}_j\cong \g{so}_3$ is the Lie algebra of the isotropy group of the isometry group of the boundary component $B_j\cong \R H^3$. For each $j\in\{1,2\}$, the action has an $11$-dimensional minimal singular orbit and can be constructed by canonical extension of the cohomogeneity one action on the boundary component $B_j\cong\R H^3$ which has a single point as singular orbit.
\item The action of the subgroup $H_{j,1}^\Lambda$, $j\in\{1,2\}$, of $G_2^\C$ with Lie algebra
\[
\g{h}_{j,1}^\Lambda=\R i H_{\alpha_{j+1}}\oplus\g{a}\oplus(\g{n}\ominus\g{g}_{\alpha_{j+1}}),
\]
where indices are taken modulo $2$ and $\R i H_{\alpha_{j+1}}\cong \g{so}_2$ is contained in the Lie algebra of the isotropy group of the isometry group of the boundary component $B_j\cong \R H^3$. For each $j\in\{1,2\}$, the action has a $12$-dimensional minimal singular orbit and can be constructed by canonical extension of the cohomogeneity one action on the boundary component $B_j\cong\R H^3$ which has a geodesic as singular orbit.
\item The action of the subgroup $H_{1,\g{v}}$ of $G_2^\C$ with Lie algebra
\[
\g{h}_{1,\g{v}}=\g{g}_{-\alpha_2}\oplus\g{g}_0\oplus\g{g}_{\alpha_2}\oplus\g{g}_{2\alpha_1+\alpha_2}\oplus\g{g}_{3\alpha_1+\alpha_2}\oplus\g{g}_{3\alpha_1+2\alpha_2},
\]
where $\g{v}=\g{g}_{\alpha_1}\oplus\g{g}_{\alpha_1+\alpha_2}$. This action has a $10$-dimensional minimal singular orbit.
\end{enumerate}
\end{theorem}

\begin{proof}
We will consider the different cases in Theorem~\ref{th:BT}. If the orbits form a Riemannian foliation, we get the actions in (1) and (2). According to \cite{BT:tohoku} the action in (3) is the only one with a totally geodesic singular orbit.

Now let us determine the actions induced by canonical extension. The symmetric space $M$ has two maximal boundary components $B_1$ and $B_2$, both isometric to $\R H^3$ with certain constant curvature metrics, but not isometric to each other because of the different lengths of the simple roots. There are, up to orbit equivalence, exactly two cohomogeneity one actions on $\R H^3$ with a singular orbit, namely the action of the isotropy group $SO_3$ (producing a point as singular orbit, and geodesic spheres around it as principal orbits), and the action of the Lie group with Lie algebra $\R H_{\alpha_{j+1}}\oplus \g{so}_2$ on $B_j\cong\R H^3$, $i\in\{1,2\}$ (producing a geodesic as singular orbit, and tubes around it as principal orbits). The canonical extensions of these actions lead to the actions in (4) and~(5). Note that the two actions in (4) (similarly with (5)) are not orbit equivalent to each other. Indeed, the normal spaces to their singular orbits are Lie triple systems which give rise to totally geodesic submanifolds having different curvatures depending on whether $j=1$ or $j=2$.

The most difficult part of the proof will consist in analysing the case of actions induced by nilpotent construction. We have to consider the two possible choices of maximal proper subsystems of $\Lambda=\{\alpha_1,\alpha_2\}$, namely $\Phi_1=\{\alpha_2\}$ and $\Phi_2=\{\alpha_1\}$.

\medskip

\emph{Nilpotent construction with $\Phi_1=\{\alpha_2\}$}. In this case we have
\begin{align*}
\g{n}^1_1&=\g{g}_{\alpha_1}\oplus\g{g}_{\alpha_1+\alpha_2}\cong\C^2,
\\
\g{l}_1&=\g{g}_{-\alpha_2}\oplus\g{g}_0\oplus\g{g}_{\alpha_2}=\g{g}_1\oplus\g{z}_1\oplus\g{a}_1\cong\g{sl}_2(\C)\oplus\g{u}_1\oplus\R \cong\g{gl}_2(\C),
\\
\g{k}_1&=\g{k}_{\alpha_2}\oplus\g{k}_0=(\g{g}_1\cap\g{k}_1)\oplus\g{z}_1\cong\g{su}_2\oplus\g{u}_1\cong\g{u}_2.
\end{align*}
The adjoint action of the subalgebra $\g{g}_1\cong\g{sl}_2(\C)$ of $\g{l}_1$ on $\g{n}_1^1\cong \C^2$ is a nontrivial complex representation. Hence, it is equivalent to the irreducible representation of $\g{sl}_2(\C)$ on $\C^2$. 
The adjoint action of the subalgebra $\g{z}_1\cong\g{u}_1 = \R iH^1$ on $\g{n}_1^1 = \g{g}_{\alpha_1}\oplus\g{g}_{\alpha_1+\alpha_2} \cong \C \oplus \C$ is the standard one.
Moreover, the action of $\g{k}_1$ on $\g{n}^1_1$ is equivalent to the standard representation of $\g{u}_2$ on $\C^2$.

We show now that, given a linear subspace $\g{v}$ of $\g{n}^1_1$, the nilpotent construction in this setting produces a cohomogeneity one action (i.e.\ $\g{v}\in\mathcal{V}$ in the notation of Section~\ref{sec:nilpotent}) if and only if $\g{v}$ is a complex subspace of $\g{n}_1^1\cong\C^2$.

We show first that $\g{v}\in\mathcal{V}$ implies that $\g{v}$ is complex. Because of condition (ii), the subspace $\g{v}$ must have constant K\"ahler angle $\varphi\in[0,\pi/2]$ (see~\cite{BB:crelle}), which implies that its dimension is even, so we can assume that $\dim\g{v}=2$. Up to the action of an element of $K_1^0\cong U_2$, we can assume that
\[
\g{v}=\mathrm{span} \{(1,0),(i\cos(\varphi),i\sin(\varphi))\},
\]
where $i$ is the imaginary unit and coordinates are with respect to some $\C$-orthonormal basis $\{e_1,e_2\}$ with $e_1\in\g{g}_{\alpha_1}$ and $e_2\in\g{g}_{\alpha_1+\alpha_2}$. Then $\g{n}^1_1\ominus\g{v}=\mathrm{span}\{(0,1),(-i\sin(\varphi),i\cos(\varphi))\}$. Assuming that $\varphi\neq 0$, some elementary calculations show that
\[
N_{\g{m}_1}(\g{n}_1^1\ominus\g{v})\cong\left\{\begin{pmatrix}a-ib\cos(\varphi) & -ib\sin(\varphi)\\ -2a\cot(\varphi)+ic & -a+i b \cos(\varphi)\end{pmatrix}:a,b,c\in\R\right\}.
\]
Taking into account that the Cartan involution $\theta$ of $\g{g}$ restricts to the standard involution of $\g{g}_1\cong\g{sl}_2(\C)$ given by minus conjugate transpose of a matrix, one can calculate that the projection of $N_{\g{m}_1}(\g{n}_1^1\ominus\g{v})$ onto $\g{p}$ is
\[
(1-\theta)N_{\g{m}_1}(\g{n}_1^1\ominus\g{v})\cong\left\{\begin{pmatrix}a & -a \cot(\varphi) - i e\\ -a\cot(\varphi)+ie & -a\end{pmatrix}:a,e\in\R\right\}.
\]
But this means that the orbit of the action of $N_{M_1}(\g{n}_1^1\ominus\g{v})$ through the origin has dimension $2$, so condition (i) fails to be true, which gives the desired contradiction.

Finally, let $\g{v}$ be a complex subspace of $\C^2$. If $\dim_\C\g{v}=2$, then (i) and (ii) in Proposition~\ref{prop:subtle} are satisfied trivially, and we obtain the cohomogeneity one action of the group $H_{1,\g{v}}$ given in (6). Assume then that $\dim_\C\g{v}=1$. In this case there is a $U_1\subset U_2\cong K_1^0$ acting transitively on the unit sphere of $\g{v}$. Moreover, $N_{\g{g}_1}(\g{n}_1^1\ominus\g{v})$ is isomorphic to the Lie subalgebra of upper triangular matrices in $\g{sl}_2(\C)$, or equivalently, to some proper parabolic subalgebra of $\g{g}_1\cong \g{sl}_2(\C)$. Hence $N_{M_1}^0(\g{n}_1^1\ominus\g{v})$ acts transitively on $B_1\cong\R H^3$. Therefore, $\g{v}$ satisfies conditions (i) and (ii) in Proposition~\ref{prop:subtle}. In this case, there exists an element in $K_1^0\cong U_2$ that maps $\g{v}$ onto $\g{g}_{\alpha_1}$, and hence the corresponding action of cohomogeneity one is orbit equivalent to the action of the group $H^\Lambda_{2,1}$ described in (5).

\medskip

\emph{Nilpotent construction with $\Phi_2=\{\alpha_1\}$}. In this case we have
\begin{align*}
\g{n}^1_2&=\g{g}_{\alpha_2}\oplus\g{g}_{\alpha_1+\alpha_2}\oplus\g{g}_{2\alpha_1+\alpha_2}\oplus\g{g}_{3\alpha_1+\alpha_2}\cong\C^4,
\\
\g{l}_2&=\g{g}_{-\alpha_1}\oplus\g{g}_0\oplus\g{g}_{\alpha_1}=\g{g}_2\oplus\g{z}_2\oplus\g{a}_2\cong\g{sl}_2(\C)\oplus\g{u}_1\oplus\R \cong\g{gl}_2(\C),
\\
\g{k}_2&=\g{k}_{\alpha_1}\oplus\g{k}_0=(\g{g}_2\cap\g{k}_2)\oplus\g{z}_2\cong\g{su}_2\oplus\g{u}_1\cong\g{u}_2.
\end{align*}
Here and in the rest of the proof, $\g{g}_2$ refers to $[\g{l}_2,\g{l}_2]$ and not to the exceptional Lie algebra of $G_2$. Then, analogously as in the previous case, the subalgebra $\g{g}_2$ of $\g{l}_2$ is isomorphic to $\g{sl}_2(\C)$ and is given by the complex span of $\{X,\theta X, H_{\alpha_1}'\}$, where $X$ is some nonzero vector in $\g{g}_{\alpha_1}$ and $H_{\alpha_1}'=\frac{2}{\langle\alpha_1,\alpha_1\rangle}H_{\alpha_1}$.
Since we have a root system of type $(G_2)$, it is easy to check that the eigenvalues of $\ad(H_{\alpha_1}')\rvert_{\g{n}^1_2}$ are $-3,-1,1$ and $3$. This means that the adjoint action of $\g{g}_2$ on $\g{n}^1_2$ is equivalent to the irreducible complex representation of $\g{sl}_2(\C)$ on $\C^4$.
Since it is of quaternionic type (it follows for example from~\cite[p.~244]{Simon}), then it is irreducible as real representation as well.
On the other hand, the action of $\g{z}_2 = \R iH^2\cong\g{u}_1$ on each root space in $\g{n}^1_2$ is the standard one; recall that $H^2\in\g{a}$ is determined by $\alpha_1(H^2)=0$ and $\alpha_2(H^2)=1$. Hence, we get that the adjoint action of $\g{k}_2$ on $\g{n}^1_2$ is equivalent to the irreducible representation $\rho_3\otimes\sigma$ of $\g{u}_2=\g{su}_2\oplus\g{u}_1$, where $\rho_3$ denotes the irreducible representation of $\g{su}_2$ on $\C^4$ and $\sigma$ is the standard action of $\g{u}_1$ on~$\C$.

Now we have to determine the subspaces $\g{v}$ of $\g{n}^1_2$ that produce cohomogeneity one actions by means of the nilpotent construction; in terms of the notation in Section~\ref{sec:nilpotent}, we have to determine $\mathcal{V}$. Since $U_2$ does not act transitively on the unit sphere of $\g{n}^1_2$, we have that $\g{n}^1_2\notin \mathcal{V}$. Proposition~\ref{prop:v} then guarantees that $\mathcal{V}=\Ad(K_2)\mathcal{V}_1$. Recall that $\mathcal{V}_1$ is the subset of subspaces $\g{v}$ in $\mathcal{V}$ such that $N_{\g{m}_2}(\g{v})$ is contained in $\g{q}_{2,1}\oplus\g{z}_2$, where we will take $\g{q}_{2,1}$ as the parabolic subalgebra $\C H_{\alpha_1}'\oplus\C X$ of $\g{g}_2\cong \g{sl}_2(\C)$. Our aim now is to determine $\mathcal{V}_1$.

Let $\g{v}\in\mathcal{V}_1$. Then $N_{\g{k}_2}(\g{v})\subset (\g{q}_{2,1}\cap\g{k}_2)\oplus\g{z}_2=\g{k}_0\cong \g{u}_1\oplus\g{u}_1$. This implies that $\dim\g{v}=2$, since $U_1\times U_1$ cannot act transitively on any sphere of dimension greater than one. Now take $i H_{\alpha_1}'$ and $iH^2$ as generators of $\g{k}_0\ominus\g{z}_2$ and $\g{z}_2$, respectively.
Then, a generic element $i(r H_{\alpha_1}' +s H^2) \in \g{k}_0$, $r,s\in\R$, acts via the adjoint representation on $\g{n}^1_2\cong \C^4$ by means of the $4\times 4$ diagonal complex matrix whose nonzero entries are $i(s-3r)$, $i(s-r)$, $i(s+r)$ and $i(s+3r)$.

Fix any nonzero vector $v\in \g{v}$. Then $\g{v}=\R v\oplus\R (r_0 \ad(iH_{\alpha_1}') +s_0 \ad(iH^2))v$, for some $r_0$, $s_0\in \R$. Let $\{e_0,e_1,e_2,e_3\}$ be a $\C$-orthonormal basis of $\g{n}^1_2\cong\C^4$, where $e_0\in\g{g}_{\alpha_2}$ and $e_j=\ad(X)^{j}e_1\in\g{g}_{\alpha_2+j\alpha_1}$ for $j=1,2,3$. Set $v=\sum_{j=0}^3 z_j e_j$, for $z_j=x_j+iy_j\in\C$.
We will show that $z_0=z_1=z_2=0$. For this purpose, let us assume that at least two coordinates of $v$ are nonzero, and we will get a contradiction.

If at least two coordinates of $v$ are nonzero, one can easily show that the normalizer $N_{\g{k}_2}(\g{v})$ must be one dimensional and, indeed, it is generated by $i(r_0 H_{\alpha_1}' + s_0 H^2)$. Moreover, for some multiple $\tilde{H}$ of $r_0 H_{\alpha_1}' +s_0 H^2$, it must happen that $\ad(i\tilde{H})$ acts diagonally on $\g{n}_2^1$ (with respect to the aforementioned basis) with coefficients $\pm i$. Let us set $\ad(i\tilde{H})e_j=i\varepsilon_j e_j$, where $\varepsilon_j\in\{\pm 1\}$, $j\in\{0,1,2,3\}$. In particular, $\g{v}=\R v \oplus \R v'$, where $v'=\sum_{j=0}^3 \varepsilon_j iz_j e_j$.

On the other hand, up to a rescaling of $X$, the adjoint action of a generic element
\[
(a+ib)H'_{\alpha_1} +(c+id)X-(e+if)\theta X + is H^2 \in \g{m}_2=\g{g}_2\oplus\g{z}_2
\]
on $\g{n}^1_2\cong\C^4$ adopts the matrix form
\[
\begin{pmatrix}-3(a+ib)+is & 3(e+if) &0 &0\\ c+id & -(a+ib)+is & 4(e+if) & 0 \\ 0 & c+id & a +ib+is & 3(e+if) \\ 0 & 0& c+id & 3(a +ib)+is\end{pmatrix}.
\]
We have to determine the real parameters $a,b,c,d,s$ such that
\[
w=\ad((a+ib)H'_{\alpha_1}+(c+id)X+is H^2)v 
 \]
belongs to $\g{v}$, i.e.\ $w=\lambda v + \mu v'$ for real numbers $\lambda$ and $\mu$. Let $k$ be the first integer for which $z_k\neq 0$. Then the condition $w=\lambda v + \mu v'$ implies that $\lambda=(2k-3)a$ and $\mu=\varepsilon_k((2k-3)b+s)$. Let $l$ be the smallest integer $l\in\{0,\dots,3\}$, $l>k$, such that $z_l\neq 0$. If $l>k+1$, then we similarly obtain that $\lambda=(2l-3)a$ and $\mu=\varepsilon_l((2l-3)b+s)$, which implies $a=0$. If $l=k+1$, then one can express $c$ and $d$ in terms of $a$ and $b((2k-3)\varepsilon_l-(2l-3)\varepsilon_k)+s(\varepsilon_l-\varepsilon_k)$. In any case, we deduce that the projection of $N_{\g{m}_2}(\g{v})=N_{\g{q}_{2,1}\oplus\g{z}_2}(\g{v})$ onto $\g{p}$ cannot have dimension~$3$, which means that condition (i) of the nilpotent construction cannot be satisfied (since $\dim B_2=\dim \R H^3=3$), thus contradicting the hypothesis $\g{v}\in\mathcal{V}_1$. Hence, we must have $\g{v}=\C e_j$, for some $j\in\{0,\dots,3\}$. However, if $j\neq 3$, the requirement $\ad((a+ib)H'_{\alpha_1}+(c+id)X+is H^2)e_j\subset \C e_j$ implies that $c=d=0$, which again leads to a contradiction with condition (i).

Therefore, we have  $\g{v}=\C e_3$, in which case $N_{\g{q}_{2,1}\oplus\g{z}_2}(\g{v})=\g{q}_{2,1}\oplus\g{z}_2$. Altogether, we deduce that the nilpotent construction method only produces a cohomogeneity one action for the choice $\g{v}=\C e_3=\g{g}_{3\alpha_1+\alpha_2}\in\mathcal{V}_1$. But since there is an element in $K_2\cong U_2$ mapping $\g{g}_{3\alpha_1+\alpha_2}$ onto $\g{g}_{\alpha_2}$, it turns out that the cohomogeneity one action of the group $H_{2,\g{g}_{3\alpha_1+\alpha_2}}$ obtained by nilpotent construction is orbit equivalent to the action of the group $H^\Lambda_{1,1}$ described in~(5). This concludes the proof.
\end{proof}

The same arguments employed above (in particular, the ones corresponding to the case $\Phi_1=\{\alpha_2\}$) can be used to obtain the classification of cohomogeneity one actions on the symmetric space $M = SL_3(\C)/SU_3$, which has rank $2$ and dimension $8$. Its root system $\Sigma$ is of type $(A_2)$ and can be identified with the root system of the complex simple Lie algebra $\g{sl}_3(\C)$. All root spaces have complex dimension $1$. Let $\Lambda=\{\alpha_1,\alpha_2\}$ be a set of simple roots, so that $\Sigma^+=\{\alpha_1,\alpha_2,\alpha_1+\alpha_2\}$. The maximal abelian subalgebra $\g{a}$ has dimension~$2$ and is spanned by the root vectors $H_{\alpha_1}$ and $H_{\alpha_2}$. Moreover, $\g{k}_0\cong\g{u}_1\oplus\g{u}_1$. Taking into account the Dynkin diagram symmetry for the root system $(A_2)$, this leads to the following classification:

\begin{theorem}\label{th:SU3}
Each cohomogeneity one action on $M=SL_3(\C)/SU_3$ is orbit equivalent to one of the following cohomogeneity one actions on $M$:
\begin{enumerate} [{\rm (1)}]
\item The action of the subgroup $H_\ell$ of $SL_3(\C)$ with Lie algebra
\[
\g{h}_\ell=(\g{a}\ominus\ell)\oplus\g{n},
\]
where $\ell$ is a one-dimensional linear subspace of $\g{a}$. The orbits are isometrically congruent to each other and form a Riemannian foliation on $M$.
\item The action of the subgroup $H_1$ of $SL_3(\C)$ with Lie algebra
\[
\g{h}_1=\g{a}\oplus(\g{n}\ominus\ell),
\]
where $\ell$ is a one-dimensional linear subspace of $\g{g}_{\alpha_1}$. The orbits form a Riemannian foliation on $M$ and there is exactly one minimal orbit.
\item The action of $SL_2(\C) \times \R \subset SL_3(\C)$, which has a totally geodesic singular orbit isometric to the symmetric space $SL_2(\C)/SU_2 \times \R \cong \R H^3 \times \R$.
\item The action of $SL_3(\R)  \subset SL_3(\C)$, which has a totally geodesic singular orbit isometric to the symmetric space $SL_3(\R)/SO_3$.
\item The action of the subgroup $H_{1,0}^\Lambda$ of $SL_3(\C)$ with Lie algebra
\[
\g{h}_{1,0}^\Lambda=\g{k}_{\alpha_2}\oplus\R i H_{\alpha_2}\oplus(\g{a}\ominus\R H_{\alpha_2})\oplus \g{g}_{\alpha_1} \oplus \g{g}_{\alpha_1 + \alpha_2},
\]
where  $\g{k}_{\alpha_2}\oplus\R i H_{\alpha_2}=\g{g}_1\cap\g{k}_1\cong \g{so}_3$ is the Lie algebra of the isotropy group of the isometry group of the boundary component $B_1\cong \R H^3$. The action has a $5$-dimensional minimal singular orbit and can be constructed by canonical extension of the cohomogeneity one action on the boundary component $B_1\cong\R H^3$ which has a single point as singular orbit.
\item The action of the subgroup $H_{1,1}^\Lambda$ of $SL_3(\C)$ with Lie algebra
\[
\g{h}_{1,1}^\Lambda=\R i H_{\alpha_2}\oplus\g{a}\oplus \g{g}_{\alpha_1} \oplus \g{g}_{\alpha_1 + \alpha_2},
\]
where $\R i H_{\alpha_2}\cong \g{so}_2$ is contained in the Lie algebra of the isotropy group of the isometry group of the boundary component $B_1\cong \R H^3$. The action has a $6$-dimensional minimal singular orbit and can be constructed by canonical extension of the cohomogeneity one action on the boundary component $B_1\cong\R H^3$ which has a geodesic as singular orbit.
\end{enumerate}
\end{theorem}

\section{The classification in the noncompact real two-plane Grassmannians}\label{sec:Grassmannian}

In this section we classify, up to orbit equivalence, the cohomogeneity one actions on the noncompact real two-plane Grassmann manifolds $SO^0_{2,n+2}/SO_2SO_{n+2}$, $n\geq 1$

The symmetric space $M=SO^0_{2,n+2}/SO_2SO_{n+2}$ has rank $2$ and dimension $2n+4$. Its root system $\Sigma$ is of type $(B_2)$. Let $\Lambda=\{\alpha_1,\alpha_2\}$, where $\alpha_1$ is the longest simple root. Then $\Sigma^+=\{\alpha_1,\alpha_2,\alpha_1+\alpha_2,\alpha_1+2\alpha_2\}$, where the multiplicities of the two long roots $\alpha_1$ and $\alpha_1+2\alpha_2$ are $1$, and those of the two short roots $\alpha_2$ and $\alpha_1+\alpha_2$ are $n$. The maximal abelian subalgebra $\g{a}$ has dimension $2$ and is spanned by the root vectors $H_{\alpha_1}$ and $H_{\alpha_2}$. Moreover, $\g{k}_0\cong \g{so}_n$ acts by the standard representation on the root spaces of dimension $n$, and trivially on those of dimension one.

\begin{theorem}\label{th:Grassmannian}
Every cohomogeneity one action on $M=SO^0_{2,n+2}/SO_2SO_{n+2}$, $n \geq 1$, is orbit equivalent to one of the following cohomogeneity one actions on $M$:
\begin{enumerate} [{\rm (1)}]
\item The action of the subgroup $H_\ell$ of $SO^0_{2,n+2}$ with Lie algebra
\[
\g{h}_\ell=(\g{a}\ominus\ell)\oplus\g{n},
\]
where $\ell$ is a one-dimensional linear subspace of $\g{a}$. The orbits are isometrically congruent to each other and form a Riemannian foliation on $M$.
\item The action of the subgroup $H_{j}$, $j\in{1,2}$, of $SO^0_{2,n+2}$ with Lie algebra
\[
\g{h}_{j}=\g{a}\oplus(\g{n}\ominus\ell_j),
\]
where $\ell_j$ is a one-dimensional linear subspace of $\g{g}_{\alpha_j}$. The orbits form a Riemannian foliation on $M$ and there is exactly one minimal orbit.
\item The action of $SO^0_{1,n+2}\subset SO^0_{2,n+2}$, which has a totally geodesic singular orbit isometric to a real hyperbolic space $\R H^{n+2}$.
\item The action of $SO^0_{2,n+1}\subset SO^0_{2,n+2}$, which has a totally geodesic singular orbit isometric to the real Grassmannian $SO^0_{2,n+1}/SO_2SO_{n+1}$.
\item If $n=2k$ is even, the action of $SU_{1,k+1}\subset SO^0_{2,2k+2}$, which has a totally geodesic singular orbit isometric to a complex hyperbolic space $\C H^{k+1}$.
\item The action of the subgroup $H_{1,k}^\Lambda$, $k\in\{0,\dots,n-1\}$, of $SO^0_{2,n+2}$ with Lie algebra
\[
\g{h}_{1,k}^\Lambda=N_{\g{k}_1}(\g{w})\oplus(\g{a}\ominus\R H_{\alpha_{2}})\oplus(\g{n}\ominus\g{g}_{\alpha_{2}})\oplus \g{w},
\]
where $\g{w}$ is a $k$-dimensional subspace of $\R H_{\alpha_2}\oplus\g{g}_{\alpha_2}$ containing $H_{\alpha_2}$ if $k\geq 1$, and $N_{\g{k}_1}(\g{w})\cong \g{so}_{n-k+1}\oplus \g{so}_k$ is the normalizer of $\g{w}$ in the Lie algebra $\g{k}_1=\g{k}_0\oplus\g{k}_{\alpha_2}\cong\g{so}_{n+1}$ of the isotropy group of the isometry group of the boundary component $B_1\cong \R H^{n+1}$. The action has a minimal singular orbit of codimension $n-k+1$ and can be constructed by canonical extension of the cohomogeneity one action on the boundary component $B_1\cong\R H^{n+1}$ which has a totally geodesic $\R H^k$ as a singular orbit.
\item The action of the subgroup $H_{2}^\Lambda$ of $SO^0_{2,n+2}$ with Lie algebra
\[
\g{h}_{2}^\Lambda=\g{k}_{\alpha_{1}}\oplus(\g{a}\ominus\R H_{\alpha_{1}})\oplus(\g{n}\ominus\g{g}_{\alpha_{1}}),
\]
where $\g{k}_{\alpha_{1}}\cong \g{so}_2$ is the Lie algebra of the isotropy group of the isometry group of the boundary component $B_2\cong \R H^2$. The action has a minimal singular orbit of codimension two and can be constructed by canonical extension of the cohomogeneity one action on the boundary component $B_2\cong\R H^2$ which has a single point as orbit.
\end{enumerate}
\end{theorem}

\begin{proof}
We consider the different cases in Theorem~\ref{th:BT}. If the orbits form a Riemannian foliation, we get the actions in (1) and (2). According to \cite{BT:tohoku}, the actions in (3), (4) and (5) are precisely those with a totally geodesic singular orbit.

Now we determine the actions induced by canonical extension. The symmetric space $M$ has two maximal boundary components $B_1\cong \R H^{n+1}$ and $B_2\cong \R H^2$. They have different curvatures because of the different lengths of the simple roots. The well-known classification of cohomogeneity one actions on real hyperbolic spaces gives then rise to the actions described in (6) and (7) via canonical extension. None of the actions in (6) is orbit equivalent to (7) due to the different constant curvature of $B_1$ and $B_2$. Moreover, none of the actions in (6) or (7) is orbit equivalent to any action in (1)-(5), because the singular orbits in (6) and (7) are minimal but not totally geodesic.

We proceed now with the investigation of the nilpotent construction method.

\medskip

\emph{Nilpotent construction with $\Phi_1=\{\alpha_2\}$}. In this case we have
\begin{align*}
\g{n}_1=\g{n}^1_1&=\g{g}_{\alpha_1}\oplus\g{g}_{\alpha_1+\alpha_2}\oplus\g{g}_{\alpha_1+2\alpha_2}\cong\R^{n+2},
\\
\g{l}_1&=\g{g}_{-\alpha_2}\oplus\g{g}_0\oplus\g{g}_{\alpha_2}=\g{g}_1\oplus\g{a}_1\cong\g{so}_{1,n+1}\oplus\R,
\\
\g{k}_1&=\g{k}_{\alpha_2}\oplus\g{k}_0=\g{so}_{n+1}.
\end{align*}
The action of $\g{m}_1=\g{g}_1$ on $\g{n}_1$ is equivalent to the standard representation of $\g{so}_{1,n+1}$ on $\R^{n+2}$. In particular, the action of $\g{k}_1$ on $\g{n}_1$ splits into a trivial one-dimensional module $\R \xi$ and a standard $\g{so}_{n+1}$-module $\xi^\perp$.

Let $\g{v}$ be a subspace of $\g{n}_1$ in the conditions of Proposition~\ref{prop:subtle}, i.e.\ $\g{v}\in \mathcal{V}$. Then $\g{v}$ must be contained in $\xi^\perp$, because otherwise there could not exist a subgroup of $K_1^0\cong SO_{n+1}$ acting transitively on the unit sphere of $\g{v}$. If $\dim \g{v}=k$, then $N_{\g{m}_1}(\g{v})$ would be isomorphic to $\g{so}_{1,n-k+1}\oplus\g{so}_k$. But then the corresponding connected subgroup of $M_1$ could not act transitively on $B_1\cong\R H^{n+1}$, which contradicts $\g{v}\in\mathcal{V}$. Hence, the nilpotent construction for the choice $\Phi_1=\{\alpha_2\}$ does not lead to any example.

\medskip

\emph{Nilpotent construction with $\Phi_2=\{\alpha_1\}$}. Now we have
\begin{align*}
\g{n}^1_2&=\g{g}_{\alpha_2}\oplus\g{g}_{\alpha_1+\alpha_2}\cong\R^{2n},
\\
\g{l}_2&=\g{g}_{-\alpha_1}\oplus\g{g}_0\oplus\g{g}_{\alpha_1}=\g{g}_2\oplus\g{a}_2\cong\g{sl}_2(\R)\oplus\g{so}_n\oplus\R,
\\
\g{k}_2&=\g{k}_{\alpha_1}\oplus\g{k}_0=\g{so}_{2}\oplus\g{so}_n.
\end{align*}
The representation of $\g{g}_2$ on $\g{n}^1_2$ is equivalent to the exterior tensor product representation of $\g{sl}_2(\R)\oplus\g{so}_n$ on $\R^2\otimes\R^n\cong\R^{2n}$. Similarly, the representation of $K_2^0 \cong SO_2 \times SO_n$ on $\g{n}^1_2$ is equivalent to the exterior tensor product representation $SO_2 \times SO_n$  on $\R^2\otimes\R^n\cong\R^{2n}$, which is also equivalent to the isotropy representation of $SO^0_{2,n}/SO_2SO_n$. 
Choose orthonormal bases $e_1,e_2$ of $\R^2$ and $f_1,\ldots,f_n$ of $\R^n$. We identify the tangent space of $SO^0_{2,n}/SO_2SO_n$ at a point $o$ with $\R^{2n} \cong \R^2 \otimes \R^n$. Then $e_i \otimes f_j$, $i \in \{1,2\}$, $j \in \{1,\ldots,n\}$, is a basis of the tangent space and a maximal flat is given by $\R (e_1 \otimes f_1) \oplus \R (e_2 \otimes f_2)$.
We can identify $\g{g}_{\alpha_2}$ with the span of $e_1 \otimes f_j$, $j \in \{1,\ldots,n\}$, and $\g{g}_{\alpha_1+\alpha_2}$ with the span of $e_2 \otimes f_j$, $j \in \{1,\ldots,n\}$. Let $T$ be a generator of $\g{k}_{\alpha_1}\cong \g{so}_2$.

Let $\g{v}\in\mathcal{V}$ and $v=v_1+v_2\in\g{v}$, $v\neq 0$, where $v_1\in\g{g}_{\alpha_2}$ and $v_2\in\g{g}_{\alpha_1+\alpha_2}$. First we will prove that if $[T,v_1]\perp v_2$, then $v_1=0$ or $v_2=0$. Under the assumption $[T,v_1]\perp v_2$, we can take the orthonormal bases  $e_1,e_2$ and $f_1,\ldots,f_n$ above such that $v_1$ is proportional to $e_1\otimes f_1$ and $v_2$ is proportional to $e_2\otimes f_2$. We can then write $v=r e_1\otimes f_1+s e_2\otimes f_2$, for some $r$, $s\in \R$. 
Let $A+S$ be an element of $N_{\g{m}_2}(\g{v})\subset \g{sl}_{2}(\R)\oplus\g{so}_n$, where $A\in\g{sl}_{2}(\R)$ and $S\in \g{so}_n$. If we put $A=\begin{pmatrix}
a & b\\ c & -a
\end{pmatrix}$, then it acts on $\g{n}_2^1$ by
\[
[A,e_1\otimes f_j]=a e_1\otimes f_j + c e_2\otimes f_j,\qquad [A,e_2\otimes f_j]=b e_1\otimes f_j - a e_2\otimes f_j, \qquad j=1,\dots,n.
\]
Then, on the one hand:
\begin{align*}\label{eq:actionAS}\nonumber
[A+S,r e_1\otimes f_1+s e_2\otimes f_2]={}&ar e_1\otimes f_1 -a s e_2\otimes f_2 +c r e_2\otimes f_1 + b s e_1\otimes f_2
\\
&+ r [S,e_1\otimes f_1]+s [S,e_2\otimes f_2]\in\g{v}.
\end{align*}
On the other hand, since $N_{\g{k}_2}(\g{v})$ acts transitively on the unit sphere of $\g{v}$, we have the orthogonal decomposition 
\[
\g{v}=\R(r e_1\otimes f_1+s e_2\otimes f_2)\oplus [N_{\g{k}_2}(\g{v}),r e_1\otimes f_1+s e_2\otimes f_2],
\]
where the second addend is always orthogonal to the vectors $e_1\otimes f_1$ and $e_2\otimes f_2$. But altogether we deduce that either $a=0$, or $r$ or $s$ must vanish. If both $r$ and $s$ are nonzero, then the projection of $N_{\g{m}_2}(\g{v})$ onto $\g{p}$ has at most dimension $1$, which implies that $N_{M_2}^0(\g{n}_{2,\g{v}})$ cannot act transitively on $B_2\cong\R H^2$, thus contradicting the assumption $\g{v}\in\mathcal{V}$. Hence, either $r=0$ or $s=0$, and the claim follows.

Now, by conjugating by an element of $K_2^0$, we can assume that there is a unit element in $\g{v}$ of the form $r e_1\otimes f_1+s e_2\otimes f_2$, for some $r$, $s\in \R$.  Because of the claim proved above, we have that either $r=0$ or $s=0$. Again, via conjugation by an element of $K_2^0$, it is not restrictive to assume that $s=0$, so that $e_1\otimes f_1\in\g{v}$.

If $T\perp N_{\g{k}_2}(\g{v})$, then $\g{v}$ can be any subspace of $\g{g}_{\alpha_2}$, thus producing one of the actions described in \cite[p.~145]{BT:crelle}, which are orbit equivalent to the actions described in (6), for $k\geq 1$.

We finally consider the case where $tT+S\in N_{\g{k}_2}(\g{v})$, for some nonzero $t\in \R$ and some $S\in \g{k}_0\cong\g{so}_n$. In this situation we have that 
\[
t e_2\otimes f_1+[S,e_1\otimes f_1]=[tT+S,e_1\otimes f_1]\in \g{v}.
\]
Since $[S,e_1\otimes f_1]\in \g{g}_{\alpha_2}$, $e_2\otimes f_1\in\g{g}_{\alpha_1+\alpha_2}$, $[T,[S,e_1\otimes f_1]]\perp e_2\otimes f_1$ and $t\neq 0$, the claim above implies that $[S,e_1\otimes f_1]=0$. Let now $S'\in N_{\g{k}_2}(\g{v})\cap\g{k}_0$. Then $tT+S+S'\in N_{\g{k}_2}(\g{v})$ and hence
\[
t e_2\otimes f_1+[S',e_1\otimes f_1]=[tT+S+S',e_1\otimes f_1]\in\g{v}
\]
But again by the claim we have that $[S',e_1\otimes f_1]=0$. Therefore we have shown that $[N_{\g{k}_2}(\g{v}), e_1\otimes f_1]=\R (e_2\otimes f_1)$, and hence, $\g{v}=\R (e_1\otimes f_1)\oplus\R (e_2\otimes f_1)$. This means that the nilpotent construction in this case produces a cohomogeneity one action with a singular orbit of codimension $2$. However, this action is orbit equivalent to the one described in~(4).
\end{proof}


\end{document}